\documentclass[11pt]{article}

\usepackage{amsmath, amssymb, amsthm}   
\usepackage{mathrsfs}                  
\usepackage{graphicx}                  
\usepackage{enumitem}                  
\usepackage{hyperref}                  
\usepackage{xcolor,tikz}                    

\usepackage{lmodern}                  
\usepackage[T1]{fontenc}
\usepackage[utf8]{inputenc}

\newtheorem{theorem}{Theorem}[section]

\newtheorem{definition}[theorem]{Definition}
\newtheorem{proposition}[theorem]{Proposition}
\newtheorem{lemma}[theorem]{Lemma}

\title{Universal Cycles for Affine Planes and 3-Subspaces over Finite Fields}
\author{Yu Hsuan Hsieh, Ming-Hsuan Kang}

\def \Gr{\mathrm{Gr}}

\begin{document}

\maketitle

\begin{abstract}
We construct universal cycles for affine planes in \(\mathbb F_q^n\) for
all prime powers \(q\) and all \(n\ge4\), using sliding windows of length
three. The construction is local-to-global: explicit local cycles are built
on frame configurations, the linear \(2\)-subspaces are organized by a
layered frame decomposition, and the resulting cycles are assembled by
gluing along shared directions. The universal cycle obtained has direction
set containing all \(1\)-subspaces. We also extend the construction to
universal cycles for \(3\)-subspaces of \(\mathbb F_q^n\).
\end{abstract}

\section{Introduction}

Universal cycles give compact cyclic encodings of large combinatorial
families. Given a finite collection $\mathcal{F}$, a \emph{universal cycle}
is a cyclic sequence whose consecutive windows encode every element of
$\mathcal{F}$ exactly once. Such cycles are cyclic analogues of exhaustive
listings and are closely related to Gray codes, de Bruijn sequences, and
Hamiltonian cycles in suitable transition graphs.

The theory of universal cycles was initiated by
Chung, Diaconis, and Graham~\cite{CDG92} and is closely related to the
broader theory of combinatorial Gray codes.  In this language, one seeks
cyclic listings of combinatorial objects in which consecutive objects differ
by a controlled local move.  This point of view has been developed for many
classes of discrete structures; see the surveys of Savage~\cite{Sav97} and
M\"utze~\cite{Mut23}.  Representative universal-cycle results include
constructions for permutations, multisets, and other combinatorial families;
see, for example, \cite{HJZ09, Joh09, Knu11}.

For vector spaces over finite fields, related Gray-code questions arise on
projective spaces and Grassmann graphs, with connections to subspace coding
and network coding; see Schwartz~\cite{Sch14}.  In this geometric setting,
a fundamental result of Jackson~\cite{Jac09} constructs universal cycles
for $2$-dimensional subspaces of $\mathbb{F}_q^n$ via a recursive
decomposition, while more recent algebraic constructions exploit the
multiplicative structure of finite fields~\cite{CKH25}.

The affine setting considered here differs from the linear Grassmannian
setting in an essential way.  A linear $2$-subspace is determined by its
directions, whereas an affine plane has both a direction space and a
translation parameter.  Thus the construction must simultaneously control
Grassmannian adjacency among directions and compatibility among parallel
classes.  This is the source of the local frame constructions, translation
lifting, and global gluing developed below.

In recent work~\cite{KangChou26}, the authors studied universal cycles for
affine lines in $\mathrm{AG}(n,q)$. That work introduced a projective
extension framework incorporating points at infinity to encode directions,
and developed a construction based on pairwise and triple configurations
combined with lifting and gluing arguments. This approach provides a
flexible mechanism for handling the parallelism inherent in affine
geometry.

This paper extends the affine-line picture to affine planes. In dimension
two, the interaction between parallel classes and affine parameters is more
intricate, and the local configurations must be organized more carefully.

\medskip
\noindent\textbf{Universal cycles for affine planes.}
Our main object is the family of affine planes in \(\mathbb F_q^n\). An
affine plane
\[
w+P,\qquad P=\langle u_1,u_2\rangle,
\]
is encoded by a sliding window of length three. A typical window has the
form
\[
(\,w,\,[u_1],\,[u_2]\,),
\]
where $[u]$ denotes the direction spanned by a nonzero vector
$u$, i.e.\ the corresponding $1$-dimensional subspace.

The local constructions also use equivalent windows in which some direction
entries are replaced by affine points, such as
\[
(\,w,\; w+u_1,\; [u_2]\,)
\quad\text{or}\quad
(\,w,\; w+u_1,\; w+u_2\,).
\]
This mixture of affine points and directions is the triple-window language
used throughout the paper.

\medskip
\noindent\textbf{Local construction via frames.}
The construction begins locally. In Section~3, the linear parts are
organized into frames: three-frames, four-frames, and a cyclic \(9\)-frame
in the characteristic-\(3\) case. Anchor parametrizations lift these linear
configurations to families of affine planes. The local cycles are built so
that shared directions become gluing interfaces; in particular, for
directions \([u]\) in the relevant direction set, the cycles contain
segments of the form
\[
[u]\to v,
\]
or the reverse orientation, with \(v\) an affine point.

\medskip
\noindent\textbf{Layered and global construction.}
The local pieces are then organized by layers. Section~4 decomposes the
linear \(2\)-subspaces into layers \(\Gr_k(2,V)\). Within each layer, pair
and triple constructions produce connected frame decompositions whose
direction sets contain all directions in \(V_k\), including the affine
directions arising from the parameters \(a+e_k\). Translation lifting
extends base decompositions to whole parallel classes.

Finally, Section~5 assembles the layer-wise decompositions. The frame graph
records when local cycles share a direction, and the connected gluing
principle joins the corresponding cycles. For \(q>2\), this is carried out
through the global frame decomposition. For \(q=2\), the construction uses
the layered decompositions for \(k\ge4\) together with an explicit cycle in
the small case \(\mathbb F_2^3\).

\medskip
\noindent\textbf{Main result.}
\begin{theorem}\label{thm:affine_ucycle}
For every prime power $q$ and every $n \ge 4$, there exists a universal
cycle for affine planes in $\mathbb{F}_q^n$ whose direction set contains
all $1$-subspaces.
\end{theorem}

\medskip
\noindent\textbf{Connections and extensions.}
The affine-plane construction also fits the usual hyperplane
decomposition: affine planes in \(\mathbb F_q^n\) correspond to
\(2\)-subspaces of \(\mathbb F_q^{n+1}\) not contained in a fixed
hyperplane. Using this viewpoint, Section~5 extends the construction to
\(3\)-subspaces.

\begin{theorem}\label{thm:3_subspace_ucycle}
For every prime power $q$ and every $n \ge 4$, there exists a universal
cycle for $3$-subspaces in $\mathbb{F}_q^n$ whose direction set contains
all $1$-subspaces.
\end{theorem}

\medskip
\noindent\textbf{Structure of the paper.}
Section~2 introduces the basic tools: anchor sets, translation, and
gluing. Section~3 develops local cycles on frame configurations. Section~4
constructs connected frame decompositions within each layer
\(\Gr_k(2,V)\). Section~5 assembles the layers, proves the affine-plane
universal cycle theorem, and extends the construction to \(3\)-subspaces.

\medskip
\noindent\textbf{Acknowledgment.}
Portions of the exposition were prepared with the assistance of ChatGPT.
The authors also performed computer-assisted checks of the local frame
constructions and gluing procedures for small prime fields and dimensions;
these checks were used to detect possible indexing errors and to confirm
the behavior of the exceptional cases.

\section{Preliminaries}

Let $V=\mathbb{F}_q^n$. Throughout the paper, we abbreviate
``$k$-dimensional subspace'' to \emph{$k$-subspace}.

\subsection{Affine Planes and Parallel Classes}

We begin with the affine objects used throughout the paper. For a linear
subspace $U\subset V$ of dimension $d\in\{1,2\}$, define its associated
\emph{parallel class} by
\[
\mathcal A(U)
=
\{\,x+U : x\in V\,\},
\]
the set of all affine subspaces with linear part $U$.

We refer to the cases $d=1$ and $d=2$ as \emph{affine lines} and
\emph{affine planes}, respectively.

\subsection{Anchor Sets and Parametrization}

Anchor sets provide canonical representatives for the affine subspaces in
a fixed parallel class.

\begin{definition}
Let $U\subset V$ be a linear subspace. A subset $R\subset V$ is called an
\emph{anchor set} for $\mathcal A(U)$ if every affine subspace
$A\in \mathcal A(U)$ intersects $R$ in exactly one point.
\end{definition}

Thus an anchor set provides a canonical \emph{anchor} for each affine
subspace in the parallel class, giving a parametrization of
\(\mathcal A(U)\).

We first record the resulting parametrization.

\begin{lemma}[Anchor parametrization]\label{lem:anchor_param}
Let $U\subset V$ and let $R$ be an anchor set for $\mathcal A(U)$. Then every
$A\in \mathcal A(U)$ can be written uniquely as
\[
A = r + U
\qquad \text{with } r\in R.
\]
Equivalently, the map
\[
R \longrightarrow \mathcal A(U),
\qquad r \longmapsto r+U,
\]
is a bijection.
\end{lemma}

The main source of anchor sets is a direct-sum complement.

\begin{lemma}\label{lem:linear_anchor}
Suppose $V=U\oplus W$. Then $W$ is an anchor set for $\mathcal A(U)$.
\end{lemma}

\subsection{Translation and Symmetry}

Translation is the basic symmetry of parallel classes. For any
$a\in V$ and any affine subspace $x+U$, we have
\[
(x+U)+a = (x+a)+U \in \mathcal A(U).
\]

It also preserves anchor sets.

\begin{lemma}[Translation of anchor sets]\label{lem:affine_anchor}
Let $R$ be an anchor set for $\mathcal A(U)$ and let $a\in V$. Then
\[
R+a := \{r+a : r\in R\}
\]
is also an anchor set for $\mathcal A(U)$.
\end{lemma}

Thus anchor parametrizations may be translated freely. This flexibility
will be used to align anchor sets across different parallel classes.

\medskip

The same symmetry applies to universal cycles.

\begin{lemma}[Translation of universal cycles]\label{lem:translate_cycle}
Let $\mathcal{A}$ be a translation-invariant collection of affine planes,
and let $\mathcal U$ be a universal cycle on $\mathcal{A}$. For any
$a\in V$, define $\mathcal U+a$ by translating all affine-point entries
while leaving direction entries unchanged. Then $\mathcal U+a$ is also a
universal cycle on $\mathcal{A}$.
\end{lemma}

This invariance is used in the gluing process to shift local cycles until
their interface segments match.

\subsection{Gluing of Universal Cycles}

We now formulate the gluing tools. The basic operation is concatenation of
linear segments; cyclic gluing is then obtained by cutting cycles at a
shared interface.

\begin{definition}
A \emph{universal segment} is a finite sequence whose length-$3$ windows
encode each element of a given collection exactly once.

For a sequence $\mathcal U=(x_1,\dots,x_m)$, define
\[
\mathrm{init}(\mathcal U) := (x_1,x_2),
\qquad
\mathrm{term}(\mathcal U) := (x_{m-1},x_m).
\]
\end{definition}

The first lemma gives the concatenation rule for compatible segments.

\begin{lemma}[Gluing of universal segments]\label{lem:gluing_segment}
Let $\mathcal U_1=(x_1,\dots,x_m)$ and $\mathcal U_2=(y_1,\dots,y_n)$ be
universal segments encoding disjoint collections $\mathcal A_1$ and
$\mathcal A_2$. Suppose
\[
\mathrm{term}(\mathcal U_1) = (x_{m-1},x_m)
= (y_1,y_2) = \mathrm{init}(\mathcal U_2).
\]
Define
\[
\mathcal U_1 \cdot \mathcal U_2
:=
(x_1,\dots,x_m,y_3,\dots,y_n),
\]
i.e., the sequence obtained by identifying the common pair
$\mathrm{term}(\mathcal U_1)=\mathrm{init}(\mathcal U_2)$ and removing the
duplicate entries.

Then $\mathcal U_1 \cdot \mathcal U_2$ is a universal segment encoding
$\mathcal A_1 \sqcup \mathcal A_2$.
\end{lemma}

\begin{proof}
This follows directly from the gluing definition and the disjointness of
the encoded collections.
\end{proof}

We next pass from segments to cycles by cutting at a common interface.

\begin{lemma}[Gluing of universal cycles]\label{lem:gluing_cycle}
Let $\mathcal A_1$ and $\mathcal A_2$ be disjoint collections with
universal cycles $\mathcal U_1$ and $\mathcal U_2$. If both cycles contain
a common consecutive pair $(a,b)$, then they can be concatenated to form a
universal cycle on $\mathcal A_1 \sqcup \mathcal A_2$.
\end{lemma}

\begin{proof}
Rotate both cycles so that $(a,b)$ is the initial consecutive pair. We cut
each cycle at this pair and view it as a universal segment whose initial
and terminal pairs are both $(a,b)$, duplicating the pair at the two ends
of the segment if necessary. The resulting segments are compatible, so they
concatenate by Lemma~\ref{lem:gluing_segment}. Interpreting the result
cyclically gives the desired universal cycle.
\end{proof}

Finally, the same argument applies to a cyclic list of compatible segments.

\begin{lemma}[Cyclic gluing]\label{lem:gluing_multiple}
Let $\mathcal U_1,\dots,\mathcal U_k$ be universal segments encoding
pairwise disjoint collections $\mathcal A_1,\dots,\mathcal A_k$. Suppose
\[
\mathrm{term}(\mathcal U_i)=\mathrm{init}(\mathcal U_{i+1})
\quad (i=1,\dots,k-1),
\qquad
\mathrm{term}(\mathcal U_k)=\mathrm{init}(\mathcal U_1).
\]
Then the concatenation
\[
\mathcal U = \mathcal U_1 \cdot \mathcal U_2 \cdots \mathcal U_k
\]
forms a universal cycle encoding $\bigsqcup_{i=1}^k \mathcal A_i$.
\end{lemma}

\begin{proof}
Repeated application of Lemma~\ref{lem:gluing_segment} gives a universal
segment encoding all $\mathcal A_i$. The final compatibility condition
closes the sequence cyclically, yielding a universal cycle.
\end{proof}

\section{Local Constructions via Frames}

This section provides the local building blocks for universal
triple-window cycles. The basic objects are small configurations of affine
planes organized as three- and four-frames. We encode them by consecutive
triple intersections, and later assemble the resulting local cycles into
global ones.

\subsection{Geometric Encoding and Realization}

We first set up the geometric language used by the local encodings.

\medskip

\noindent\textbf{Affine span with directions.}
We use one notation for affine points and directions. For entries
$x_1,\dots,x_k$, each either a point or a $1$-subspace, define
\[
\langle x_1,\dots,x_k\rangle_{\mathrm{aff}}
\]
to be the smallest affine subspace containing all affine points among the
$x_i$ and whose direction space contains all directions among the $x_i$.

\medskip

\noindent\textbf{Intersection with directions.}
We also extend intersection so that it may return either an affine point
or a direction. Let
\[
A = w + U,\qquad B = w' + V
\]
be affine subspaces of $V$.

\begin{itemize}
\item If $A \cap B \neq \varnothing$, then their intersection is the usual
affine intersection
\[
A \cap B.
\]
\item If $A \cap B = \varnothing$, then we define their intersection to be
the linear subspace
\[
U \cap V,
\]
viewed as a direction.
\end{itemize}

In particular, when $U \cap V$ is $1$-dimensional, the intersection is a
direction.

\medskip

This convention matches the two kinds of entries that occur in the local
cycles: affine points when subspaces meet, and directions otherwise.
It allows a single triple of entries to describe an affine plane even when
some adjacent planes meet only through their linear directions.

\medskip

\noindent\textbf{Triple-transversal sequences.}
We next record the condition under which consecutive triples have entries
of this controlled form.
Let
\[
\Pi_i = w_i + P_i \subset \mathbb{F}_q^n, \qquad i=1,\dots,m,
\]
be a sequence of affine planes. We say that $(\Pi_i)$ is
\emph{triple-transversal} if each consecutive triple intersection
\[
\Pi_i \cap \Pi_{i+1} \cap \Pi_{i+2}
\]
is a single element, either an affine point or a direction.

\noindent\textbf{Encoding and admissibility.}
Once the triple intersections are controlled, the question is whether they
recover the original planes by affine span.
A sequence $v=(v_1,\dots,v_{m+2})$ with affine-point or direction entries
is an \emph{encoding} of $(\Pi_i)$ if
\[
\Pi_i = \langle v_i, v_{i+1}, v_{i+2}\rangle_{\mathrm{aff}},
\qquad i=1,\dots,m.
\]
We call a triple-transversal sequence \emph{admissible} if such an
encoding exists.

When $(\Pi_i)$ is triple-transversal, the interior entries are uniquely
determined by
\[
v_i = \Pi_i \cap \Pi_{i-1} \cap \Pi_{i-2},
\qquad i=3,\dots,m.
\]
Thus only the boundary entries remain to be chosen:
\[
v_1,\, v_2,\, v_{m+1},\, v_{m+2},
\]
and they must be chosen so that
\[
\Pi_i = \langle v_i, v_{i+1}, v_{i+2}\rangle_{\mathrm{aff}}
\]
holds for all $i$.

Consequently, the boundary data
\[
\mathrm{init}(v):=(v_1,v_2),
\qquad
\mathrm{term}(v):=(v_{m+1},v_{m+2}),
\]
determine the encoding, and different choices may give distinct but
equivalent encodings of the same sequence.

\medskip

\noindent\textbf{General position.}
The preceding encoding depends on triple intersections having controlled
size. We record the simple linear condition that guarantees this behavior
in the frame configurations used below.

Informally, subspaces are in \emph{general position} when they intersect
as transversally as possible, meaning that all intersections have the
smallest dimensions compatible with the ambient space.

Here we use this notion only for triples of \(2\)-subspaces inside a
common \(3\)-subspace. This is the local linear situation that will occur
whenever three consecutive affine planes are encoded by a single
triple-window entry.

\begin{definition}[General position]
Three \(2\)-subspaces
\[
P_1,P_2,P_3 \subset U,
\qquad \dim U = 3,
\]
are said to be in \emph{general position} if
\[
\dim(P_i\cap P_j)=1
\quad\text{for all } i\ne j,
\]
and
\[
P_1\cap P_2\cap P_3 = \{0\}.
\]
Equivalently, every pair intersects in a line, while no nonzero vector
lies in all three planes simultaneously.

\end{definition}

We use the same terminology for affine planes in the following sense.
Three affine planes lying in a common affine \(3\)-subspace are said to be
in \emph{general position} if their direction spaces are in general
position. In this case, the three affine planes intersect in a unique
affine point.

The local constructions require this condition not just for one triple but
for several triples inside a small family. Frames package that requirement
for the linear parts that support the affine-plane blocks.

\begin{definition}[Frames]
A \emph{frame} is a finite unordered family of \(2\)-subspaces
\[
\mathcal P=\{P_1,\dots,P_m\}\subset V,
\qquad m\ge3,
\]
such that every triple of distinct planes in \(\mathcal P\) is in general
position.

The \emph{direction set} of \(\mathcal P\) is
\[
\mathcal D(\mathcal P)
=
\{P_i\cap P_j : i\ne j\}.
\]

A frame with \(m=3\) (resp.\ \(m=4\)) is called a \emph{three-frame}
(resp.\ \emph{four-frame}).
\end{definition}

The direction set records the possible linear interfaces between local
blocks. In later gluing arguments, shared elements of these direction sets
will be used to match local cycles.

Although the definition does not require the ambient \(3\)-subspaces
associated with different triples to coincide, the next lemma shows that
all planes in a frame automatically lie in a common \(3\)-subspace. For
three-frames and four-frames, this also yields a common anchor space used
in the local constructions below.

\begin{lemma}\label{lem:frame_common_3space_anchor}
Let \(\mathcal P\) be a frame. Then all planes in \(\mathcal P\) lie in a
common \(3\)-subspace \(U\subset V\).

If moreover \(\mathcal P\) is a three-frame or a four-frame, then there
exists \(w\in U\) not contained in any \(P_i\), and for any such \(w\)
there exists a subspace \(W_{\mathcal P}\) containing \(w\) such that
\[
V=P_i\oplus W_{\mathcal P}
\qquad\text{for all } i.
\]
\end{lemma}
\begin{proof}
Choose distinct \(P_1,P_2\in\mathcal P\) and set
\(
U:=P_1+P_2
\).
Since \(\dim(P_1\cap P_2)=1\), we have \(\dim U=3\).

For \(i\ge3\), the lines \(P_i\cap P_1\) and \(P_i\cap P_2\) are
distinct. Otherwise,
\(
P_i\cap P_1\cap P_2\neq\{0\},
\)
contrary to general position. Hence these two lines span \(P_i\), so
\(P_i\subset U\). Thus all planes of \(\mathcal P\) lie in the common
\(3\)-subspace \(U\).

Now suppose \(\mathcal P\) is a three-frame or a four-frame. Since the
planes in \(\mathcal P\) contain at most \(4q-2\) one-dimensional
subspaces of \(U\), while \(U\) contains \(q^2+q+1\), there exists
\(w\in U\) not contained in any \(P_i\).

Let \(L=\langle w\rangle\). Since \(w\notin P_i\) for every \(i\),
we have \(L\not\subset P_i\) for every \(i\). Hence \(P_i\cap L=\{0\}\),
and since \(\dim P_i=2\) and \(\dim U=3\), we obtain
\(
U=P_i\oplus L
\)
for every \(i\).

Choose a complement \(V=U\oplus W_0\) and define
\(
W_{\mathcal P}:=L\oplus W_0
\).
Then
\[
V=P_i\oplus W_{\mathcal P}
\qquad\text{for all } i.
\]
\end{proof}

The preceding lemma supplies the common ambient and anchor spaces needed
for the three-frame and four-frame constructions. We now introduce the
ordered version used in the admissibility criterion.

\begin{definition}[Cycle frames]
A \emph{cycle frame} is a cyclically ordered family
\(
\mathcal P=(P_1,\dots,P_m)
\)
of \(2\)-subspaces, with \(m\ge3\), such that every three consecutive planes
\(
P_i,P_{i+1},P_{i+2}
\)
(indices modulo \(m\)) are in general position.

Throughout a cycle frame, indices are read modulo \(m\) with values in
\(\{1,\dots,m\}\). For each \(i\), define
\[
[u_i]:=P_{i-1}\cap P_i.
\]
Choose a nonzero representative \(u_i\) of each direction \([u_i]\). Then
\(
P_i=\langle u_i,u_{i+1}\rangle
\)
for all \(i\in\mathbb Z/m\mathbb Z\).

The set
\[
\mathcal D(\mathcal P)=\{[u_1],\dots,[u_m]\}
\]
is called the \emph{direction set} of \(\mathcal P\).
\end{definition}
\medskip

Cycle frames are designed for sequences: consecutive triples of linear
parts have the same controlled geometry throughout the cycle. With such an
order fixed, the remaining issue is to decide when affine translates of
these planes are recovered from the corresponding triple intersections.

\medskip

\noindent\textbf{Admissibility over a cyclic frame.}
We now connect the cyclic ordering of the linear parts with the
triple-window encoding of affine translates.

\begin{theorem}\label{thm:admissibility}
Let
\(
\mathcal P=(P_1,\dots,P_m)
\)
be a cycle frame. Let \(r\ge1\), and let
\(
(\Pi_1,\dots,\Pi_{rm})
\)
be affine planes of the form
\[
\Pi_i=w_i+P_{i\bmod m},
\]
for shifts \(w_i\in V\).
Here \(P_{i\bmod m}\) is interpreted using the modulo convention above.

Assume that every three consecutive affine planes
\(
\Pi_{i-2},\Pi_{i-1},\Pi_i
\)
lie in a common affine \(3\)-subspace. Define \((v_1,\dots,v_{rm+2})\) by
\[
v_1:=P_m\cap P_1,\quad
v_2:=P_1\cap P_2,\quad
v_{rm+1}:=P_{m-1}\cap P_m,\quad
v_{rm+2}:=P_m\cap P_1,
\]
and for \(i=3,\dots,rm\),
\[
v_i:=\Pi_i\cap \Pi_{i-1}\cap \Pi_{i-2}.
\]

If \(v_i\neq v_{i+1}\) for all \(i=1,\dots,rm+1\), then \((v_i)\) is an
encoding of \((\Pi_i)\).
\end{theorem}

\begin{proof}
For \(i\ge3\), the planes
\(
\Pi_{i-2},\Pi_{i-1},\Pi_i
\)
lie in a common affine \(3\)-subspace. Since their direction spaces are
three consecutive planes in the cycle frame \(\mathcal P\), they are in
general position. Hence
\(
v_i=\Pi_i\cap \Pi_{i-1}\cap \Pi_{i-2}
\)
is a single affine point.

We now recover \(\Pi_i\) from the window
\(
(v_i,v_{i+1},v_{i+2})
\).
Since \(v_i,v_{i+1}\in \Pi_i\cap\Pi_{i-1}\) and
\(v_i\neq v_{i+1}\), the line
\(
\langle v_i,v_{i+1}\rangle_{\mathrm{aff}}
\)
has direction \(P_{i-1}\cap P_i\). Similarly,
\(
\langle v_{i+1},v_{i+2}\rangle_{\mathrm{aff}}
\)
has direction \(P_i\cap P_{i+1}\).

At the boundary windows, the same directions are supplied directly by the
prescribed direction entries
\(
v_1,v_2,v_{rm+1},v_{rm+2}
\).

These two directions are distinct since \(\mathcal P\) is a cycle frame.
Hence
\(
\langle v_i,v_{i+1},v_{i+2}\rangle_{\mathrm{aff}}
\)
is a \(2\)-dimensional affine subspace contained in \(\Pi_i\). Therefore
\[
\langle v_i,v_{i+1},v_{i+2}\rangle_{\mathrm{aff}}
=
\Pi_i,
\]
so \((v_i)\) encodes \((\Pi_i)\).
\end{proof}
\medskip

\subsection{Three-Frame Constructions}

We begin with three-frames, where the boundary matching is most transparent.
Let
\(
\mathcal P=\{P_1,P_2,P_3\}
\)
be a three-frame with direction set
\[
\mathcal D(\mathcal P)=\{[u_1],[u_2],[u_3]\},
\]
where
\(
[u_i]=P_{i-1}\cap P_i
\)
and
\(
P_i=\langle u_i,u_{i+1}\rangle
\)
(indices modulo \(3\)). Set
\(
U=\langle u_1,u_2,u_3\rangle.
\)

By Lemma~\ref{lem:frame_common_3space_anchor}, there exists a common anchor set
\(
W_{\mathcal P}\subset V
\)
such that
\(
V=P_i\oplus W_{\mathcal P}
\)
for all \(i\).

\medskip
\noindent\textbf{Basic blocks.}

The first blocks encode one translate of the three-frame at a time.
For $i\in\mathbb Z/3\mathbb Z$, consider the planes
\[
\Pi_1=P_i,\quad \Pi_2=P_{i+1},\quad \Pi_3=P_{i+2},
\]
which lie in $U$ and, with this cyclic order, satisfy the hypotheses of
Theorem~\ref{thm:admissibility}. The theorem gives the segment
\[
B_i(0):
\quad
(P_i\cap P_{i+1})\to(P_{i+1}\cap P_{i+2})\to 0
\to(P_{i+2}\cap P_i)\to(P_i\cap P_{i+1}).
\]
For $w\in W_{\mathcal P}$, define
\[
B_i(w):=B_i(0)+w,
\]
where affine entries are translated and direction entries are unchanged.
Then $B_i(w)$ encodes exactly $w+P_i,\;w+P_{i+1},\;w+P_{i+2}$. We
summarize the resulting properties as follows.

\begin{lemma}\label{lem:block_properties}
For each $w\in W_{\mathcal P}$ and $i\in\mathbb Z/3\mathbb Z$, the block
$B_i(w)$ is a universal segment encoding $w+\mathcal P$. Moreover, for all
$w,w'\in W_{\mathcal P}$,
\[
\mathrm{term}(B_i(w))=\mathrm{init}(B_{i+2}(w')).
\]
\end{lemma}

\medskip
\noindent\textbf{Refinement.}
The basic blocks concatenate directly only when the anchor set splits
evenly among the three boundary types. When this fails, we replace one
basic block by a longer segment with the same boundary pairs.

Choose
\[
w^* \in U \setminus \bigcup_{i=1}^3 P_i,
\qquad
w^* = w_1^* + w_2^* + w_3^*,\ \ w_i^* \in \langle u_i \rangle.
\]
Then each $w_i^* \neq 0$. By Lemma~\ref{lem:frame_common_3space_anchor},
we may assume that
\[
w^* \in W_{\mathcal P}.
\]

At the origin, consider the sequence
\[
P_i,\; P_{i+1},\; P_{i+2},\;
w^*+P_i,\; w^*+P_{i+1},\; w^*+P_{i+2},
\]
which lies in $U$ and whose linear parts, in the displayed cyclic order,
satisfy the hypotheses of
Theorem~\ref{thm:admissibility}. The theorem gives a segment
\[
\begin{aligned}
R_i(0,w^*):\quad
&(P_i\cap P_{i+1})\to (P_{i+1}\cap P_{i+2})
\to 0 \to w_{i}^*\\
&\to w_i^*+w_{i+1}^* \to w^*
\to (P_{i+2}\cap P_i)\to (P_i\cap P_{i+1}).
\end{aligned}
\]
Since $w_i^* \neq 0$ for all $i$, all consecutive entries are distinct, so
Theorem~\ref{thm:admissibility} applies.

For $w\in W_{\mathcal P}$, define
\[
R_i(w,w^*):=R_i(0,w^*)+w.
\]
By Lemma~\ref{lem:translate_cycle}, translation preserves admissibility
and the encoded affine planes. The refinement has the following boundary
compatibility.
 
\begin{lemma}\label{lem:refine_cyclic}
Each $R_i(w,w^*)$ is a universal segment encoding
\[
(w+\mathcal P)\;\cup\;((w+w^*)+\mathcal P),
\]
and satisfies
\[
\mathrm{init}(R_i(w,w^*))=\mathrm{init}(B_i(w)),\qquad
\mathrm{term}(R_i(w,w^*))=\mathrm{term}(B_i(w)).
\]
\end{lemma}
\medskip

\noindent\textbf{Universal cycles.}

The local segments are now assembled by selecting blocks whose boundary
pairs form an Eulerian circuit.

\begin{theorem}\label{thm:cyclic_parallel_classes}
Suppose $(q,n)\neq(2,3)$, and let $\mathcal P$ be a three-frame with
anchor set $W_{\mathcal P}$. Then
\[
\bigsqcup_{w\in W_{\mathcal P}} (w+\mathcal P)
=
\mathcal A(P_1)\sqcup \mathcal A(P_2)\sqcup \mathcal A(P_3)
\]
admits a universal triple-window cycle in which, for every
$[u]\in\mathcal D(\mathcal P)$, there exists an affine point $v$ such
that the segment $[u]\to v$ appears.
\end{theorem}

\begin{proof}
The proof has two parts. First we choose one block for each translate so
that the boundary pairs form a balanced directed graph; then an Eulerian
circuit in this graph gives the required cyclic concatenation.

Let $\mathcal P=\{P_1,P_2,P_3\}$ be a three-frame with direction set
\[
\mathcal D(\mathcal P)=\{P_i\cap P_{i+1}: i\in \mathbb Z/3\mathbb Z\},
\]
and let $U$ be the common $3$-subspace containing $\mathcal P$. By
construction, each translate $w+\mathcal P$ is encoded by the basic blocks
$B_i(w)$.

\medskip
\noindent\textbf{Step 1: Block selection.}

For each $w\in W_{\mathcal P}$, the three blocks
\[
B_1(w),\ B_2(w),\ B_3(w)
\]
all encode the same translate $w+\mathcal P$. Thus we select exactly one
block for each $w$.

\medskip
\noindent\textbf{Step 2: Eulerian assembly.}

Let $G$ be the directed multigraph whose vertex set is
\[
V(G)=\{(P_i\cap P_{i+1},\,P_{i+1}\cap P_{i+2}): i\in \mathbb Z/3\mathbb Z\},
\]
and where each selected block $B_i(w)$ contributes an edge
\[
(P_i\cap P_{i+1},\,P_{i+1}\cap P_{i+2})
\longrightarrow
(P_{i+2}\cap P_i,\,P_i\cap P_{i+1}).
\]
By Lemma~\ref{lem:block_properties}, these pairs are the initial and
terminal vertices of each block.

\medskip
\noindent\textbf{Case 1: $3\mid |W_{\mathcal P}|$.}

Write $|W_{\mathcal P}|=3m$ and partition
\[
W_{\mathcal P}=W_0\sqcup W_1\sqcup W_2,
\qquad |W_0|=|W_1|=|W_2|=m.
\]
For $w\in W_j$, select the block $B_j(w)$.

Each vertex in $G$ has indegree and outdegree equal to $m$, so $G$ is
balanced. Moreover, the transition $i\mapsto i+2$ generates
$\mathbb Z/3\mathbb Z$, so $G$ is strongly connected. Hence $G$ admits an
Eulerian circuit, which gives a cyclic concatenation of the selected
blocks.

\medskip
\noindent\textbf{Case 2: $|W_{\mathcal P}|\equiv k \pmod 3$, $k\in\{1,2\}$.}

Write $|W_{\mathcal P}|=3m+k$ and choose
\[
w^*\in U\setminus\bigcup_{i=1}^3 P_i.
\]
By Lemma~\ref{lem:frame_common_3space_anchor}, after choosing the common anchor
space appropriately, we may assume $w^*\in W_{\mathcal P}$.

\smallskip
If $k=1$, choose $w_1\in W_{\mathcal P}$. Apply Case~1 to
\[
W_{\mathcal P}\setminus\{w_1+w^*\},
\]
and let $i_0$ be the index assigned to $w_1$. Replace the block
$B_{i_0}(w_1)$ by the refinement block $R_{i_0}(w_1,w^*)$.

\smallskip
If $k=2$, choose $w_1,w_2\in W_{\mathcal P}$ such that
\[
w_1\neq w_2\pm w^*.
\]
Since $(q,n)\neq(2,3)$, we have $|W_{\mathcal P}|\ge 4$, so such a choice
is possible. Then
\[
\{w_1,w_2\}\cap\{w_1+w^*,\,w_2+w^*\}=\varnothing.
\]
Apply Case~1 to
\[
W_{\mathcal P}\setminus\{w_1+w^*,\,w_2+w^*\},
\]
and for each $j=1,2$, let $i_j$ be the index assigned to $w_j$. Replace
$B_{i_j}(w_j)$ by $R_{i_j}(w_j,w^*)$.

\medskip
By Lemma~\ref{lem:refine_cyclic}, each refinement block has the same
initial and terminal vertices as the block it replaces. Thus $G$ remains
balanced and strongly connected, and still admits an Eulerian circuit.

\medskip
\noindent\textbf{Step 3: Conclusion.}

The resulting cyclic concatenation encodes each translate $w+\mathcal P$
exactly once. Each block $B_i(w)$ and $R_i(w,w^*)$ contains a segment of
the form
\[
(P_i\cap P_{i+1}) \to v
\]
for some affine point $v$, so every direction in $\mathcal D(\mathcal P)$
appears as required.
\end{proof}

\subsection{Four-Frame Constructions}

We next adapt the same local strategy to four-frames. The only new issue
is that a direct translate of all four planes has a degeneracy, so one
plane is shifted before the blocks are formed.

Let $\mathcal P'=\{P_1,P_2,P_3,Q\}$ be a four-frame, and set
\[
\mathcal P:=\{P_1,P_2,P_3\}.
\]
Then $\mathcal P$ is a three-frame, and $\mathcal P'$ is contained in a
common $3$-subspace $U$. Let $W_{\mathcal P'}$ be a common linear anchor
set. We begin with a canonical model.

\medskip

\noindent\textbf{Normal form.}
\begin{proposition}[Four-frame normal form]\label{prop:four_plane_model}
Let $\mathcal P=\{P_1,P_2,P_3\}$ be a three-frame, and suppose that
$\mathcal P\cup\{Q\}$ forms a four-frame. Then there exist linearly
independent $u_1,u_2,u_3\in V$ such that
\[
P_i=\langle u_i,u_{i+1}\rangle,\quad i=1,2,3,
\qquad
Q=\{x_1u_1+x_2u_2+x_3u_3:\ x_1+x_2+x_3=0\},
\]
and
\[
\mathcal D(\mathcal P\cup\{Q\})
=
\{[u_1],[u_2],[u_3],[u_1-u_2],[u_2-u_3],[u_3-u_1]\}.
\]
\end{proposition}

\begin{proof}
Let $U=\langle u_1,u_2,u_3\rangle$ be the common $3$-subspace. Since $Q$
meets each $P_i$ in a line, we have $Q\subset U$, hence
$Q=\ker(\ell)$ for some nonzero linear functional $\ell:U\to\mathbb F_q$.
Rescaling $u_i$ so that $\ell(u_i)=1$ yields
\[
Q=\{x_1u_1+x_2u_2+x_3u_3:\ x_1+x_2+x_3=0\}
=\langle u_1-u_2,\;u_2-u_3\rangle.
\]
The direction set follows from pairwise intersections.
\end{proof}

\noindent
For the rest of this subsection, we fix such a choice of
$u_1,u_2,u_3$ and work in this normal form.

\medskip
\noindent\textbf{Choice of the shift parameter.}
Unlike in the three-frame case, the planes in $w+\mathcal P'$ all pass
through the same affine point $w$, so a direct encoding degenerates. To
avoid this, we introduce a shift parameter $w^*\in U$ and encode the
modified family
\[
(w+\mathcal P)\;\cup\;((w-w^*)+Q).
\]
We fix such a $w^*$ once and for all. The precise conditions on $w^*$ will
be specified as needed below. By Lemma~\ref{lem:frame_common_3space_anchor},
we may assume $w^*\in W_{\mathcal P'}$.

With this shift in place, the basic blocks are defined exactly as before,
but with $Q$ translated by $-w^*$.
For $i\in\mathbb Z/3\mathbb Z$, consider the sequence at the origin
\[
P_i,\quad -w^*+Q,\quad P_{i+1},\quad P_{i+2},
\]
which lies in $U$ and whose linear parts, in the displayed cyclic order,
satisfy the hypotheses of
Theorem~\ref{thm:admissibility}. The theorem gives a segment $B_i^{(4)}(0)$.
For $w\in W_{\mathcal P'}$, define
\[
B_i^{(4)}(w):=B_i^{(4)}(0)+w.
\]

To describe this segment explicitly, write
\[
w^*=a_i u_i+a_{i+1}u_{i+1}+a_{i+2}u_{i+2},
\qquad
s=a_i+a_{i+1}+a_{i+2}.
\]
Then
\[
B_i^{(4)}(w):
\quad
(P_{i+2}\cap P_i)
\to (P_i\cap Q)
\to w-s u_{i+1}
\to w-s u_{i+2}
\to (P_{i+1}\cap P_{i+2})
\to (P_{i+2}\cap Q).
\]

The final entry is chosen as $P_{i+2}\cap Q$ so that the terminal
pair matches the initial pair of $B_{i+2}^{(4)}(w')$. This adjustment only
affects the initial triple and does not change the encoded planes.

The consecutive entries are distinct provided $s\neq 0$, equivalently
$w^*\notin Q$. Under this condition, Theorem~\ref{thm:admissibility}
applies.

\begin{lemma}\label{lem:block4_properties}
For each $w\in W_{\mathcal P'}$ and $i\in\mathbb Z/3\mathbb Z$, the block
$B_i^{(4)}(w)$ is a universal segment encoding
\[
(w+\mathcal P)\;\cup\;((w-w^*)+Q).
\]
Moreover, for all $w,w'\in W_{\mathcal P'}$,
\[
\mathrm{term}(B_i^{(4)}(w))
=
\mathrm{init}(B_{i+2}^{(4)}(w')).
\]
\end{lemma}

\medskip

\noindent\textbf{Refinement for four-frames.}

As in the three-frame case, the refinement is designed to encode two
adjacent basic blocks without changing the boundary pairs. It suffices to
construct it at the origin. We interlace the
two basic blocks corresponding to $0$ and $w^*$ via the sequence
\[
\begin{aligned}
&P_i,\quad -w^*+Q,\quad P_{i+2},\quad P_{i+1},\quad Q,\\
&w^*+P_i,\quad w^*+P_{i+1},\quad w^*+P_{i+2}.
\end{aligned}
\]
Taking consecutive triple intersections for this interlaced sequence yields
\[
\begin{aligned}
R_i^{(4)}(0,w^*):\quad
&(P_{i+2}\cap P_i)
\to (P_i\cap Q)
\to -s u_i
\to -s u_{i+2}\\
&\to 0
\to -a_{i+2}u_{i+1}+a_{i+2}u_{i+2}\\
&\to a_i u_i-(a_i+a_{i+2})u_{i+1}+a_{i+2}u_{i+2}\\
&\to w^*
\to (P_{i+1}\cap P_{i+2})
\to (P_{i+2}\cap Q).
\end{aligned}
\]

The affine entries are consecutive distinct provided
\[
a_i\neq 0,\qquad a_{i+2}\neq 0,\qquad s\neq 0.
\]
Such a choice of $w^*$ exists for every $q$: if
$\operatorname{char}(\mathbb F_q)\neq 3$, we may take
$a_1=a_2=a_3=1$; if $\operatorname{char}(\mathbb F_q)=3$, we may take
$a_1=a_2=1$ and $a_3=2$.

Since the ordering above is not a simple frame order,
Theorem~\ref{thm:admissibility} does not apply directly. A direct
verification of the consecutive length-$3$ windows shows that this segment
encodes precisely the planes encoded by the two basic blocks
$B_i^{(4)}(0)$ and $B_i^{(4)}(w^*)$. Hence, with this choice of $w^*$, the
segment is admissible and has the required boundary pairs.

\medskip
For general $w\in W_{\mathcal P'}$, define
\[
R_i^{(4)}(w,w^*):=R_i^{(4)}(0,w^*)+w.
\]
By Lemma~\ref{lem:translate_cycle}, this translates the encoded planes.
Thus $R_i^{(4)}(w,w^*)$ encodes the union of the two basic blocks
$B_i^{(4)}(w)$ and $B_i^{(4)}(w+w^*)$, namely
\[
(w+\mathcal P)\cup\{(w-w^*)+Q\}
\cup
((w+w^*)+\mathcal P)\cup\{w+Q\}.
\]
Moreover, it has the same boundary pairs as $B_i^{(4)}(w)$:
\[
\mathrm{init}(R_i^{(4)}(w,w^*))
=
\mathrm{init}(B_i^{(4)}(w)),
\qquad
\mathrm{term}(R_i^{(4)}(w,w^*))
=
\mathrm{term}(B_i^{(4)}(w)).
\]

\begin{lemma}[Local detour / refinement]\label{lem:refine_four}
For every prime power $q$, the segment $R_i^{(4)}(w,w^*)$ has the same
initial and terminal pairs as $B_i^{(4)}(w)$, and its length-$3$ windows
encode
\[
(w+\mathcal P)\cup\{(w-w^*)+Q\}
\cup
((w+w^*)+\mathcal P)\cup\{w+Q\}.
\]
\end{lemma}
\medskip

\noindent\textbf{Universal cycles.}

With the four-frame blocks and refinements in place, the universal cycle is
obtained by the same boundary-pair assembly used for three-frames.

\begin{theorem}\label{thm:four_parallel_classes}
Suppose $(n,q)\neq(3,2)$ and let $\mathcal P'=\{P_1,P_2,P_3,Q\}$ be a
four-frame. Then
\[
\bigsqcup_{P\in\mathcal P'} \mathcal A(P)
\]
admits a universal triple-window cycle in which, for every
$[u] \in\mathcal D(\mathcal P')$, there exists an affine point $v$ such that
either $[u]\to v$ or $v\to [u]$ appears.
\end{theorem}

\begin{proof}
The proof follows the Eulerian assembly scheme used for three-frames. The
only change is that the boundary pairs and the local blocks now come from
the four-frame construction: the basic blocks $B_i^{(4)}(w)$ and the
refinement blocks $R_i^{(4)}(w,w^*)$.

For each $w\in W_{\mathcal P'}$, we select one block among
\[
B_i^{(4)}(w),\qquad i\in\mathbb Z/3\mathbb Z,
\]
replacing it by $R_i^{(4)}(w,w^*)$ when refinement is required.

By Lemmas~\ref{lem:block4_properties} and~\ref{lem:refine_four}, each
selected block has boundary transition
\[
(P_{i+2}\cap P_i,\; P_i\cap Q)
\longrightarrow
(P_{i+1}\cap P_{i+2},\; P_{i+2}\cap Q).
\]
Thus the selected blocks define a directed multigraph whose vertices are
these boundary pairs, and whose edges are the selected blocks.

As before, the indices induce the cyclic transition
$i\mapsto i+2$ on $\mathbb Z/3\mathbb Z$. Hence, after choosing the blocks
with the same distribution as in the three-frame construction, the
resulting directed multigraph is balanced and strongly connected. It
therefore admits an Eulerian circuit, and concatenating the blocks along
this circuit gives a universal triple-window cycle.

The basic and refinement blocks encode exactly the corresponding affine
planes in
\[
\bigsqcup_{P\in\mathcal P'}\mathcal A(P),
\]
so every affine plane is encoded exactly once. Finally, by construction,
each direction $[u]$ in $\mathcal D(\mathcal P')$ appears in some block in a
segment of the form $[u]\to v$ or $v\to [u]$ for an affine point $v$. This
proves the required direction property.
\end{proof}

\subsection{A Cyclic $9$-Frame Construction for $q=3$}

We close the local section with the characteristic-$3$ cyclic construction,
where the block selection already balances the boundary pairs and no
refinement is needed.

\begin{theorem}\label{thm:9_parallel_classes}
Suppose \(\operatorname{char}(\mathbb F_q)=3\), and let
\(
\mathcal P=(P_1,\dots,P_9)
\)
be a cycle frame with direction set
\(
\mathcal D(\mathcal P)=\{[u_1],\dots,[u_9]\},
\)
so that
\(
P_i=\langle u_i,u_{i+1}\rangle
\)
for all \(i\in\mathbb Z/9\mathbb Z\).

Assume there exists a subspace
\(
W_{\mathcal P}\subset V
\)
such that
\(
V=P_i\oplus W_{\mathcal P}
\)
for all \(i\). Then
\[
\mathcal A(P_1)\sqcup\cdots\sqcup\mathcal A(P_9)
\]
admits a universal triple-window cycle in which, for every
\(
[u]\in\mathcal D(\mathcal P),
\)
there exists an affine point \(v\) such that the segment
\(
[u]\to v
\)
appears.
\end{theorem}

\begin{proof}
We adapt the construction of Theorem~\ref{thm:cyclic_parallel_classes} to
cyclic \(9\)-frames over fields of characteristic \(3\).

\noindent\textbf{Step 1: Local blocks.}

For $w\in W_{\mathcal P}$ and $i\in \mathbb Z/9\mathbb Z$, define
\[
\tilde{B}_i(w):
\quad
[u_i]\to [u_{i+1}]\to w\to [u_{i+2}]\to [u_{i+3}].
\]
Each block $\tilde{B}_i(w)$ encodes the three planes
\[
w+P_i,\quad w+P_{i+1},\quad w+P_{i+2}.
\]

For fixed $w$, the three blocks
\[
\tilde{B}_i(w),\quad \tilde{B}_{i+3}(w),\quad \tilde{B}_{i+6}(w)
\]
are disjoint and together encode
\[
w+P_1,\dots,w+P_9.
\]

\medskip
\noindent\textbf{Step 2: Block selection.}

Since $\operatorname{char}(\mathbb{F}_q)=3$, we may partition
\[
W_{\mathcal P}=W_0\sqcup W_1\sqcup W_2,
\qquad |W_0|=|W_1|=|W_2|.
\]
For $w\in W_j$, select the three blocks
\[
\tilde{B}_j(w),\quad \tilde{B}_{j+3}(w),\quad \tilde{B}_{j+6}(w).
\]

\medskip
\noindent\textbf{Step 3: Eulerian assembly.}

Let $G$ be the directed multigraph with vertex set
\[
V(G)=\{([u_i],[u_{i+1}]): i\in \mathbb Z/9\mathbb Z\},
\]
and where each block $\tilde{B}_i(w)$ contributes an edge
\[
([u_i],[u_{i+1}]) \longrightarrow ([u_{i+2}],[u_{i+3}]).
\]

By construction, each vertex has equal indegree and outdegree, so $G$ is
balanced. Moreover, the transition $i\mapsto i+2$ generates
$\mathbb Z/9\mathbb Z$, so $G$ is strongly connected. Hence $G$ admits an
Eulerian circuit.

\medskip
\noindent\textbf{Step 4: Conclusion.}

Concatenating the selected blocks along this circuit yields a universal
cycle. Each affine plane appears exactly once, and each block contains a
segment of the form $[u_i]\to v$, so every direction in
$\mathcal D(\mathcal P)$ appears as required.
\end{proof}

\section{Frame Decompositions in a Layer}

This section constructs connected frame decompositions for the layers
$\Gr_k(2,V)$. In each layer, we partition the relevant \(2\)-subspaces into
three-frames and four-frames while keeping the frame graph connected and
the direction set complete.

The argument combines the local pair and triple constructions with a
translation-based covering procedure. These layer-wise decompositions will
be the input for the global gluing step in the next section.

\subsection{Gluing Principle and Frame Graph}

\begin{definition}[Adjacency]
Two frames $\mathcal P$ and $\mathcal Q$ (including cyclic $9$-frames) are
\emph{adjacent} if
\[
\mathcal D(\mathcal P)\cap \mathcal D(\mathcal Q)\ne\emptyset.
\]
\end{definition}

\begin{definition}[Frame graph]
Let $\mathcal F$ be a collection of frames. The \emph{frame graph}
$G(\mathcal F)$ has vertex set $\mathcal F$, with an edge between
$\mathcal P$ and $\mathcal Q$ whenever they are adjacent.
\end{definition}

This graph records exactly where local cycles can be glued.

\begin{theorem}[Connected gluing principle]\label{thm:connected_gluing_principle}
Let $\mathcal F$ be a collection of frames that partitions the relevant
$2$-subspaces. If the frame graph $G(\mathcal F)$ is connected, then the
corresponding local universal cycles can be glued to form a global
universal cycle.
\end{theorem}

\begin{proof}
For each frame $\mathcal P$, let $\mathcal U_{\mathcal P}$ be the local
universal cycle. By construction, for every $[u]\in\mathcal D(\mathcal P)$,
the cycle contains a segment of the form
\[
[u]\to v \quad \text{or} \quad v\to [u].
\]

Start from a frame $\mathcal P_0$ and set
$\mathcal U:=\mathcal U_{\mathcal P_0}$.

Suppose $\mathcal U$ encodes a connected subcollection $\mathcal F_0$.
Let $\mathcal Q\notin\mathcal F_0$ be adjacent to $\mathcal F_0$, so that
\[
[u]\in \mathcal D(\mathcal P)\cap \mathcal D(\mathcal Q)
\]
for some $\mathcal P\in\mathcal F_0$.

Then $\mathcal U$ contains either $[u]\to v$ or $v\to [u]$, and
$\mathcal U_{\mathcal Q}$ contains a corresponding segment involving $[u]$.
By reversing $\mathcal U_{\mathcal Q}$ if necessary, we may match the
orientation. This reversal preserves the encoded affine planes, since the
affine span of a length-three window is independent of the order of its
entries. Translating $\mathcal U_{\mathcal Q}$ so that the affine points
agree, the two cycles share a common consecutive pair and can be glued.

The resulting cycle encodes $\mathcal F_0\cup\{\mathcal Q\}$ and still
contains the required interface segments. Iterating over the connected
frame graph gives the desired global universal cycle.
\end{proof}

\newcommand{\parhead}[1]{\medskip\noindent\textbf{#1}}

\subsection{Layered Decomposition and Translation Framework}

We next organize the $2$-subspaces of $V=\mathbb{F}_q^n$ into layers and
record the translation framework used to build decompositions inside a
fixed layer.

\medskip

\parhead{Layered model.}
Fix a basis $e_1,\dots,e_n$ and set $V_k:=\langle e_1,\dots,e_k\rangle$.
A $2$-subspace $P$ has \emph{index $k$} if
\[
P\subset V_k \quad \text{but} \quad P\not\subset V_{k-1}.
\]
Let $\Gr_k(2,V)$ denote the set of such subspaces. Then
\[
\Gr(2,V)=\bigsqcup_{k=2}^n \Gr_k(2,V).
\]

For $k\ge 3$, write $V_k=V_{k-1}\oplus\langle e_k\rangle$. Every
$P\in\Gr_k(2,V)$ can be written as
\[
P=\langle u, a+e_k\rangle,
\qquad
u\in V_{k-1}\setminus\{0\},\ a\in V_{k-1}.
\]
We call $[u]$ the \emph{linear direction}, $a$ the \emph{affine parameter},
and $[a+e_k]$ the \emph{affine direction}. 

\medskip

\parhead{Translation action.}
For $x\in V_{k-1}$, define
\[
T_x\bigl(\langle u, a+e_k\rangle\bigr)
=
\langle u, a+x+e_k\rangle.
\]
This action fixes linear directions and translates affine parameters.

For a family $\mathcal S$, write
\[
\mathcal D(\mathcal S)
=
\mathcal D_{\mathrm{lin}}(\mathcal S)
\sqcup
\mathcal D_{\mathrm{aff}}(\mathcal S),
\]
where
\[
\mathcal D_{\mathrm{lin}}(\mathcal S)=\{[u]\},\qquad
\mathcal D_{\mathrm{aff}}(\mathcal S)=\{[a+e_k]\}.
\]
Under translation, $\mathcal D_{\mathrm{lin}}$ is invariant, while
$\mathcal D_{\mathrm{aff}}$ is shifted by $x$.

\medskip

\parhead{Parallel classes.}
For a $1$-subspace $L=\langle u\rangle\subset V_{k-1}$, define
\[
\mathcal A(L)
=
\{\,\langle u, a+e_k\rangle : a \in V_{k-1}\,\}.
\]
By Lemma~\ref{lem:linear_anchor}, any complement
$V_{k-1}=L\oplus W$ yields an anchor parametrization
\[
\mathcal A(L)=\{\,\langle u, a+e_k\rangle : a\in W\,\},
\]
in which the affine parameter $a$ is uniquely determined.

\parhead{Two-level decomposition.}
When several parallel classes share the same anchor set \(W\), we use a
second decomposition of \(W\) to separate base data from translations.

\begin{definition}[Two-level anchor decomposition]\label{def:two_level_anchor}
Let $\mathcal A(L_1),\dots,\mathcal A(L_r)$ admit a common anchor set
$W\subset V_{k-1}$. Set
\[
U_0 := \mathrm{Span}(L_1,\dots,L_r)\cap W,
\qquad
W = U_0 \oplus W_0.
\]
We call this a \emph{two-level anchor decomposition}.
\end{definition}

Then each affine parameter decomposes uniquely as $a=u_0+b$, with
$u_0\in U_0$ and $b\in W_0$. Thus it is enough to work on the base subset
supported on $U_0\oplus\langle e_k\rangle$ and then translate it by
elements of \(W_0\).

\medskip

\parhead{Translation lifting.}
The following lemma records this reduction.

\begin{lemma}[Translation lifting]\label{lem:translation_lifting}
Let $\mathcal A(L_1),\dots,\mathcal A(L_r)$ admit a two-level anchor
decomposition $W=U_0\oplus W_0$, and define
\[
\mathfrak X
=
\{\,\langle u_i, u_0+e_k\rangle : 1\le i\le r,\ u_0\in U_0\,\}.
\]
Then
\[
\mathcal A(L_1)\sqcup \cdots \sqcup \mathcal A(L_r)
=
\bigsqcup_{b\in W_0} T_b(\mathfrak X).
\]

If $\mathcal S_{\mathfrak X}$ is a frame partition of $\mathfrak X$, then
\[
\mathcal S
=
\bigsqcup_{b\in W_0} T_b(\mathcal S_{\mathfrak X})
\]
is a frame partition of $\mathcal A(L_1)\sqcup \cdots \sqcup \mathcal A(L_r)$.

Moreover,
\[
\mathcal D_{\mathrm{lin}}(\mathcal S)
=
\mathcal D_{\mathrm{lin}}(\mathcal S_{\mathfrak X}),
\]
and affine directions are translated by $b$. If
$\mathcal S_{\mathfrak X}$ is connected and
$\mathcal D_{\mathrm{lin}}(\mathcal S_{\mathfrak X})\neq\varnothing$,
then $\mathcal S$ is connected.
\end{lemma}

\begin{proof}
Every affine parameter $a\in W$ decomposes uniquely as $a=u_0+b$, giving
\[
\langle u_i, a+e_k\rangle
=
T_b(\langle u_i, u_0+e_k\rangle).
\]
This gives the stated decomposition, and the remaining assertions follow
from translation invariance.
\end{proof}

\subsection{Local Frame Constructions in a Layer}

We now build the local decompositions used inside a layer. By
Lemma~\ref{lem:translation_lifting}, each construction may be carried out
on the base family $\mathfrak X$ supported on
$U_0\oplus\langle e_k\rangle$ and then extended by translation.

\subsubsection{Pair Construction}

The pair construction is the basic local move inside a layer. It handles
two parallel classes with directions in a common \(2\)-subspace and
produces frames that already carry the linear directions of those classes.

\begin{proposition}[Pair construction]\label{prop:pair_construction}
Let $L_1=\langle u\rangle$ and $L_2=\langle v\rangle$ be distinct
$1$-subspaces contained in a $2$-subspace $U\subset V_{k-1}$.
Then $\mathcal A(L_1)\sqcup \mathcal A(L_2)$ admits a frame partition
$\mathcal S$ such that
\[
\mathcal D_{\mathrm{lin}}(\mathcal S)=\{[u],[v]\},
\]
and
\[
\mathcal D_{\mathrm{aff}}(\mathcal S)
\supset
\begin{cases}
\{[w+e_k]: w\in W\}, & q \text{ even},\\[1ex]
\{[w+e_k]: w\in W\setminus (a_0+W_0)\}, & q \text{ odd},
\end{cases}
\]
for some two-level anchor decomposition
\[
W=\langle a_0\rangle\oplus W_0.
\]
Moreover, $\mathcal S$ is connected when $q\neq 3$. When $q=3$,
$\mathcal S$ has exactly two connected components, whose affine
parameter sets contain $W_0$ and $-a_0+W_0$, respectively.
\end{proposition}

\begin{proof}
Choose a $1$-subspace $L_0=\langle a_0\rangle\subset U$ distinct from
$L_1$ and $L_2$. Then
\[
U=L_1\oplus L_0=L_2\oplus L_0.
\]
Fix a decomposition
\[
V_{k-1}=U\oplus W_0,
\]
and set
\[
W:=L_0\oplus W_0.
\]
Then $W$ is a common anchor set, giving the two-level decomposition
\[
W=L_0\oplus W_0.
\]

Define
\[
\mathfrak X
=
\{\langle u,a+e_k\rangle,\ \langle v,a+e_k\rangle : a\in L_0\}.
\]
By Lemma~\ref{lem:translation_lifting}, it is enough to partition the base
family $\mathfrak X$.

\medskip
\noindent\textbf{Case 1: $q$ odd.}
Define the three-frames
\[
\mathcal T_+
=
\{\langle u,e_k\rangle,\ \langle v,e_k\rangle,\
  \langle u,a_0+e_k\rangle\},
\]
\[
\mathcal T_-
=
\{\langle u,-a_0+e_k\rangle,\
  \langle v,a_0+e_k\rangle,\
  \langle v,-a_0+e_k\rangle\},
\]
and for
\[
a\in L_0\setminus\{0,\pm a_0\}
\]
modulo the relation $a\sim -a$, define
\[
\mathcal Q_a
=
\{\langle u,\pm a+e_k\rangle,\
  \langle v,\pm a+e_k\rangle\}.
\]
These frames partition $\mathfrak X$. Moreover,
\[
\mathcal D_{\mathrm{lin}}(\mathcal S_{\mathfrak X})
=
\{[u],[v]\},
\]
and
\[
\mathcal D_{\mathrm{aff}}(\mathcal S_{\mathfrak X})
\supset
\{[a+e_k]: a\in L_0\setminus\{a_0\}\}.
\]

If $q>3$, some four-frame $\mathcal Q_a$ intersects both
$\mathcal T_+$ and $\mathcal T_-$, so
$\mathcal S_{\mathfrak X}$ is connected.

When $q=3$, the only nonzero elements of $L_0$ are $\pm a_0$, so no
four-frame $\mathcal Q_a$ occurs. Hence $\mathcal T_+$ and
$\mathcal T_-$ form two connected components. Moreover,
\[
\mathcal D_{\mathrm{aff}}(\mathcal T_+)\supseteq \{[e_k]\},
\qquad
\mathcal D_{\mathrm{aff}}(\mathcal T_-)\supseteq \{[-a_0+e_k]\}.
\]
\medskip
\noindent\textbf{Case 2: $q$ even.}
For $a\in L_0$ modulo the relation $a\sim a+a_0$, define
\[
\mathcal Q_a
=
\{
\langle u,a+e_k\rangle,\
\langle u,a+a_0+e_k\rangle,\
\langle v,a+e_k\rangle,\
\langle v,a+a_0+e_k\rangle
\}.
\]
These four-frames partition $\mathfrak X$. Also,
\[
\mathcal D_{\mathrm{lin}}(\mathcal S_{\mathfrak X})
=
\{[u],[v]\},
\]
and
\[
\mathcal D_{\mathrm{aff}}(\mathcal S_{\mathfrak X})
=
\{[a+e_k]: a\in L_0\}.
\]
Since consecutive four-frames overlap, the partition
$\mathcal S_{\mathfrak X}$ is connected.

\medskip
Applying Lemma~\ref{lem:translation_lifting} gives the required partition
\(\mathcal S\).
\end{proof}

\subsubsection{Triple Construction}

Pairing accounts for most parallel classes. When an odd number remains, we
need one construction that treats three parallel classes at once while
preserving the direction information needed for later gluing.

\begin{proposition}[Triple construction]\label{prop:triple_construction}
Let $k\ge 4$ and let $L_i=\langle u_i\rangle$ for $i=1,2,3$ be distinct
$1$-subspaces contained in a $2$-subspace $U\subset V_{k-1}$. Then
\[
\mathcal A(L_1)\sqcup \mathcal A(L_2)\sqcup \mathcal A(L_3)
\]
admits a frame partition $\mathcal S$ such that
\[
\mathcal D_{\mathrm{lin}}(\mathcal S)=\{[u_1],[u_2],[u_3]\}.
\]
If $q=3$, then $\mathcal S$ is connected and
\[
\mathcal D_{\mathrm{aff}}(\mathcal S)\supset \{[w+e_k]:w\in W\}.
\]
\end{proposition}

\begin{proof}
Write $V_{k-1}=U\oplus W_0$, and let $\mathfrak X$ be the subset of
$\mathcal A(L_1)\sqcup \mathcal A(L_2)\sqcup \mathcal A(L_3)$ supported on
$U\oplus\langle e_k\rangle$. Then
\[
\mathcal A(L_1)\sqcup \mathcal A(L_2)\sqcup \mathcal A(L_3)
=
\bigsqcup_{b\in W_0} T_b(\mathfrak X),
\]
so it remains to partition $\mathfrak X$.

\medskip
\noindent\textbf{Case 1: $q=2$.}
In this case $U$ contains exactly the three lines $L_1,L_2,L_3$, so no
additional line in $U$ is available to form a two-level anchor decomposition.
We therefore construct the partition directly.

The set $\mathfrak X$ consists of six planes and admits a partition
\[
T_1=\{\langle u_1,e_k\rangle,\ \langle u_1,u_2+e_k\rangle,\ \langle u_2,e_k\rangle\},
\]
\[
T_2=\{\langle u_3,e_k\rangle,\ \langle u_3,u_1+e_k\rangle,\ \langle u_2,u_3+e_k\rangle\},
\]
with $\mathfrak X=T_1\sqcup T_2$. Each of $T_1,T_2$ is a three-frame, and
\[
\mathcal D_{\mathrm{lin}}(T_1\cup T_2)=\{[u_1],[u_3]\}.
\]

A second partition $\mathfrak X=T_1'\sqcup T_2'$ is obtained by interchanging
$u_1$ and $u_2$, yielding
\[
\mathcal D_{\mathrm{lin}}(T_1'\cup T_2')=\{[u_2],[u_3]\}.
\]
Since $k\ge 4$, we have $|W_0|\ge 2$, and applying these two decompositions
to distinct translates ensures that all three directions
$[u_1],[u_2],[u_3]$ occur.

For the remaining translates, choose either of the two decompositions
arbitrarily. Since the translates \(T_b(\mathfrak X)\) are disjoint, this
assigns a frame partition to every translate and hence gives a frame
partition of the full union of the three parallel classes.

\medskip
\noindent\textbf{Case 2: $q\ge 3$.}
Choose a $1$-subspace $L_0=\langle a_0\rangle\subset U$ with
$L_0\notin\{L_1,L_2,L_3\}$, and set $W=L_0\oplus W_0$. Then
\[
\mathfrak X
=
\{\langle u_i,a+e_k\rangle : i=1,2,3,\ a\in L_0\}.
\]

\smallskip
\noindent\emph{Subcase $q=3$.}
Here $|L_0|=3$. After rescaling $u_1,u_2$, we may assume
\[
u_3=u_1+u_2,
\qquad
a_0=-u_1+u_2.
\]
Then $\mathfrak X$ consists of $9$ planes arranged in the cyclic order
\[
\begin{aligned}
&\langle u_1,e_k\rangle \to \langle u_2,e_k\rangle \to \langle u_2,a_0+e_k\rangle \\
&\to \langle u_3,a_0+e_k\rangle \to \langle u_3,2a_0+e_k\rangle \to \langle u_1,a_0+e_k\rangle \\
&\to \langle u_1,2a_0+e_k\rangle \to \langle u_3,e_k\rangle \to \langle u_2,2a_0+e_k\rangle .
\end{aligned}
\]
Each consecutive triple forms a three-frame, yielding a cyclic $9$-frame.
Hence $\mathcal S_{\mathfrak X}$ is connected and
\[
\mathcal D_{\mathrm{lin}}(\mathcal S_{\mathfrak X})
=
\{[u_1],[u_2],[u_3]\},
\qquad
\mathcal D_{\mathrm{aff}}(\mathcal S_{\mathfrak X})
\supseteq
\{[a+e_k]:a\in L_0\}.
\]

\smallskip
\noindent\emph{Subcase $q>3$.}
Write $q=4m+r$ with $r\in\{0,1,3\}$ and choose $R\subset L_0$ with $|R|=r$.

\smallskip
\noindent\emph{Three-frames.}
For $a\in R$, define
\[
\mathcal F_a
=
\{\langle u_1,a+e_k\rangle,\ \langle u_2,a+e_k\rangle,\ \langle u_3,a+a_0+e_k\rangle\}.
\]

\smallskip
\noindent\emph{Four-frames.}
Let $D=L_0\setminus R$ and partition $D=D_1\sqcup D_2$ with $|D_1|=|D_2|$.
Since $U=L_1\oplus L_0=L_2\oplus L_0=L_3\oplus L_0$, the parametrization
of $\mathcal A(L_3)$ differs by a translation along $L_0$, realized by
the shift $a\mapsto a_0+a$. The remaining planes decompose into
\[
\mathcal A(L_1,D_1)\sqcup \mathcal A(L_2,D_1),\quad
\mathcal A(L_2,D_2)\sqcup \mathcal A(L_3,a_0+D_2),\quad
\mathcal A(L_3,a_0+D_1)\sqcup \mathcal A(L_1,D_2).
\]
Each pair consists of two parallel classes indexed by a set of even
cardinality; pairing anchors $a$ with $-a$ (or $a$ with $a+a_0$ in
characteristic $2$) yields a partition into four-frames.

Since $D\neq\varnothing$, at least one four-frame appears, ensuring
\[
\mathcal D_{\mathrm{lin}}(\mathcal S_{\mathfrak X})
=
\{[u_1],[u_2],[u_3]\}.
\]
\end{proof}

\subsection{Connected Frame Decomposition of $\Gr_k(2,V)$}

We now assemble the local constructions into a connected core whose affine
directions cover all parameters in \(V_{k-1}\). This core is then extended
to a full decomposition of the layer.

The first step is to build a connected family that sees every affine
direction in the layer. Once such a core is available, all other frames can
be attached to it through shared directions.

\begin{theorem}[Affine covering and connectivity]
\label{thm:affine_covering}
Let $k\ge 4$, and fix a decomposition
\[
V_{k-1}=\langle \ell\rangle \oplus W.
\]
Then there exists a connected family $\mathcal S_{\mathrm{core}}$ of frames
such that
\[
\mathcal D_{\mathrm{aff}}(\mathcal S_{\mathrm{core}})
=
\{[x+e_k]:x\in V_{k-1}\}.
\]
\end{theorem}

\begin{proof}
We first cover all affine parameters and then connect the resulting
families in the frame graph.

\medskip
\noindent \textbf{Step 1: Covering $W$.}
Write $W = \langle a_0 \rangle \oplus W_0$ and fix $w_0 \in W_0\setminus\{0\}$.
For $i,j \in \mathbb F_q$, define
\[
L_{i,j} = \langle \ell + i a_0 + j w_0 \rangle.
\]
Then each $L_{i,j}$ is not contained in $W$, and for fixed $j$, the lines
$\{L_{i,j}\}_{i\in\mathbb F_q}$ are distinct. Moreover, for $i\neq i'$,
\[
\mathrm{Span}(L_{i,j},L_{i',j})
=
\langle \ell + j w_0,\ a_0\rangle,
\qquad
V_{k-1} = \mathrm{Span}(L_{i,j},L_{i',j}) \oplus W_0.
\]
Thus $W=\langle a_0\rangle \oplus W_0$ provides a common two-level anchor
decomposition for every pair $(L_{i,j},L_{i',j})$, and each such pair
satisfies the hypotheses of Proposition~\ref{prop:pair_construction}.
Fix $j\in\mathbb F_q$ and construct a family $\mathcal S_j$:

\begin{itemize}
\item If $q$ is even, one pair $(L_{0,j},L_{1,j})$ yields a connected family
$\mathcal S_j$ on
\[
\mathcal A_j := \mathcal A(L_{0,j}) \sqcup \mathcal A(L_{1,j}),
\]
covering all affine parameters in $W$.

\item If $q=3$, applying the triple construction to
$\{L_{0,j},L_{1,j},L_{2,j}\}$ yields a connected family $\mathcal S_j$ on
\[
\mathcal A_j := \mathcal A(L_{0,j}) \sqcup \mathcal A(L_{1,j}) \sqcup \mathcal A(L_{2,j}),
\]
covering all affine parameters in $W$.

\item If $q$ is odd and $q>3$, each pair $(L_{i,j},L_{i',j})$ covers all
affine parameters in $W$ except one coset of $W_0$. For the pair
$(L_{0,j},L_{1,j})$, use the auxiliary representative $a_0$, so the
missing coset is $a_0+W_0$. For the pair $(L_{2,j},L_{3,j})$, use the
auxiliary representative $2a_0$, so the missing coset is $2a_0+W_0$.
These cosets are distinct because \(a_0\notin W_0\). Hence the two pair
constructions together cover $W$. Since they share affine directions on the
remaining cosets, the union is connected, yielding a family
$\mathcal S_j$ on
\[
\mathcal A_j := \mathcal A(L_{0,j}) \sqcup \mathcal A(L_{1,j}) \sqcup \mathcal A(L_{2,j}) \sqcup \mathcal A(L_{3,j}).
\]
\end{itemize}

In all cases,
\[
\mathcal D_{\mathrm{aff}}(\mathcal S_j)
\supset
\{[w+e_k]: w\in W\}.
\]

\medskip
\noindent
\textbf{Step 2: Extending to all affine parameters.}
Since each $\mathcal A_j$ is a union of parallel classes, for each
$j\in\mathbb F_q$ we apply the translation $T_{j\ell}$ to $\mathcal S_j$.
This preserves each parallel class (and hence $\mathcal A_j$), while shifting
affine parameters from $W$ to the coset $W+j\ell$. Thus
$T_{j\ell}(\mathcal S_j)$ is again a frame partition on $\mathcal A_j$, and
\[
\mathcal D_{\mathrm{aff}}(T_{j\ell}(\mathcal S_j))
\supset
\{[w+j\ell+e_k]: w\in W\}.
\]

Since
\[
V_{k-1}
=
\bigsqcup_{j\in\mathbb F_q} (W+j\ell),
\]
the families $\{T_{j\ell}(\mathcal S_j)\}_{j\in\mathbb F_q}$ together cover
\(\bigsqcup_{j}\mathcal A_j\) and hence all affine parameters:
\[
\mathcal D_{\mathrm{aff}}\!\left(\bigcup_{j} T_{j\ell}(\mathcal S_j)\right)
=
\{[x+e_k]: x\in V_{k-1}\}.
\]
Each $T_{j\ell}(\mathcal S_j)$ is connected, so at this stage there are at
most \(q\) connected components.

\medskip
\noindent
\textbf{Step 3: Connecting the components.}
Choose distinct $1$-subspaces $L_1,L_2\subset W$ and a complement
$W_0'$ such that
\[
V_{k-1}=L_1\oplus L_2\oplus W_0',
\qquad
\ell\in W_0'.
\]
Then $W_0'$ serves as the second component of a two-level anchor
decomposition associated with the pair $(L_1,L_2)$.

Applying Proposition~\ref{prop:pair_construction} to $(L_1,L_2)$ yields
a frame partition $\mathcal S'$. We include all of $\mathcal S'$ in the
core.  When $q\neq 3$, the family $\mathcal S'$ is connected.  When
$q=3$, the family $\mathcal S'$ has two connected components: one whose
affine parameters contain $W_0'$, and one whose affine parameters contain
a coset of the form $-a_0+W_0'$, where $a_0\in W$ is the auxiliary
representative used in the pair construction.

In every case, the component of $\mathcal S'$ containing $W_0'$ satisfies
\[
\mathcal D_{\mathrm{aff}}
\supset
\{[x+e_k]: x\in W_0'\}.
\]
For odd \(q>3\), this follows from Proposition~\ref{prop:pair_construction}
because, for the auxiliary representative \(a_0\) chosen in that
application, the only omitted coset is the nonzero coset \(a_0+W_0'\).

Since $\ell\in W_0'$, every coset $W+j\ell$ intersects $W_0'$.
Therefore, for each $j$, the translated component
\[
T_{j\ell}(\mathcal S_j)
\]
shares an affine direction with the component of $\mathcal S'$ containing
$W_0'$.  If $q=3$, the second component of $\mathcal S'$ also attaches to
the translated families: since $a_0\in W$, every coset $W+j\ell$ also
intersects $-a_0+W_0'$.  Hence all components of $\mathcal S'$ are adjacent
to the previously constructed families.

Combining these families yields a connected family
$\mathcal S_{\mathrm{core}}$ with the desired affine-direction coverage.
\end{proof}

The covering theorem supplies the connected core. The remaining task is to
turn this core into a partition of the whole layer without losing
connectivity or direction completeness.

\begin{theorem}[Connected frame decomposition of a layer]\label{thm:index-k-frame-decomposition}
Let $V=\mathbb F_q^n$ with $n \ge k\ge 4$.
Then $\Gr_k(2,V)$ admits a partition $\mathcal S$ into three-frames and
four-frames such that
\begin{itemize}
\item the frame graph is connected, and
\item $\mathcal D(\mathcal S)$ consists of all directions in $V_k$.
\end{itemize}
\end{theorem}

\begin{proof}
We start from the connected affine-direction core supplied by the previous
theorem. The rest of the proof attaches the remaining parallel classes to
this core while keeping the number of unused classes even, except in the
small \(q=2\) case where the core itself is adjusted.

Let
\[
V_{k-1} = \langle \ell \rangle \oplus W
\]
be as in Theorem~\ref{thm:affine_covering}.

\medskip
\noindent
\textbf{Step 1: Core construction.}
By Theorem~\ref{thm:affine_covering}, there exists a connected family
$\mathcal S_{\mathrm{core}}$ such that
\[
\mathcal D_{\mathrm{aff}}(\mathcal S_{\mathrm{core}})
=
\{[x+e_k] : x \in V_{k-1}\}.
\]
Thus the core already contains all affine directions.

\medskip
\noindent
\textbf{Step 2: Pairing and triple reduction.}
The number of parallel classes in $\Gr_k(2,V)$ is
\[
N_k=1+q+\cdots+q^{k-2}.
\]
By Theorem~\ref{thm:affine_covering}, the core construction uses
\[
2q+2,\qquad 3q+2,\qquad 4q+2
\]
parallel classes, according as $q$ is even, $q=3$, or $q>3$ is odd.

If the number of remaining parallel classes is even, we pair them and apply
Proposition~\ref{prop:pair_construction} to each pair. Each such
construction produces affine directions already contained in
$\mathcal D_{\mathrm{aff}}(\mathcal S_{\mathrm{core}})$, so
the resulting frames are adjacent to the core family.

It remains to consider when the number of remaining classes is odd. Since
\[
N_k\equiv
\begin{cases}
1 \pmod 2, & q \text{ even},\\
k-1 \pmod 2, & q \text{ odd},
\end{cases}
\]
the odd-remainder cases are handled as follows.

\begin{itemize}
\item If $q$ is even and $q>2$, choose unused directions
\[
L_{\alpha,0},\quad L_{\alpha,1},\quad L_{\alpha,\alpha},
\qquad
\alpha\in\mathbb F_q\setminus\mathbb F_2,
\]
where $L_{i,j}$ is as in Theorem~\ref{thm:affine_covering}. These three
directions lie in a common $2$-subspace. Applying the triple construction
removes the parity obstruction.

\item If $q=3$ and $k$ is odd, then
\[
\dim W=k-2\ge 3.
\]
Only two directions in $W$ have been used by the connecting family.
Choose a $2$-subspace $U\subset W$ containing at most one of them. Since
$U$ contains $q+1=4$ one-dimensional subspaces, at least three remain
unused.

\item If $q>3$ is odd and $k$ is even, then $\dim W=k-2\ge 2$.
Choose any $2$-subspace $U\subset W$. At most two of its directions have
already been used, while $U$ contains $q+1\ge 6$ directions. Hence at
least three directions in $U$ remain unused.

\item If $q=2$, there may be no unused triple outside the core.
We therefore modify Step~3 of the core construction by using the triple
construction instead of the pair construction. This uses three parallel
classes rather than two, so the number of remaining classes becomes even.

More explicitly, after Step~2 there are two components corresponding to
the two cosets
\[
W,\qquad \ell+W
\]
of $W$ in $V_{k-1}$. Choose a $2$-subspace $U\subset W$ with three
one-dimensional subspaces
\[
L_1,\ L_2,\ L_3.
\]
Let \(\mathfrak X\) denote the planes in
\[
\mathcal A(L_1)\sqcup \mathcal A(L_2)\sqcup \mathcal A(L_3)
\]
supported on \(U\oplus\langle e_k\rangle\).
The triple construction yields two distinct decompositions of $\mathfrak X$
with linear directions $\{[u_1],[u_3]\}$ and $\{[u_2],[u_3]\}$, whose affine
parameters lie in $W$. Apply the first decomposition to $\mathfrak X$, and
the second to $T_\ell(\mathfrak X)$.
Since $\ell \notin W$, it is transversal to $W$ and hence to $U\subset W$.
Thus we may choose a two-level anchor decomposition whose second component
contains $\ell$. Then the translation $T_\ell$, applied to the local partition
$\mathcal S_{\mathfrak X}$ as in the translation lifting construction,
shifts affine parameters from $W$ to $\ell+W$.

Thus the two resulting families contain frames, with affine parameters in
$W$ and $\ell+W$, that share the common direction $[u_3]$. These adjacent
frames connect the two components. The remaining unused parallel classes
can then be paired and handled by Proposition~\ref{prop:pair_construction}.
\end{itemize}

\medskip
\noindent
\textbf{Step 4: Direction completeness.}
Each parallel class is used exactly once, so every linear direction in
$V_{k-1}$ appears in $\mathcal S$. Together with the affine directions from
$\mathcal S_{\mathrm{core}}$, we obtain
\[
\mathcal D(\mathcal S)
=
\{\text{all directions in } V_k\}.
\]

\medskip
\noindent
Combining the preceding steps, the partition $\mathcal S$ is connected and
has the full direction set.
\end{proof}

\section{Global Assembly and Universal Cycles}

We now assemble the layer-wise constructions into a universal cycle for all
affine planes in \(V=\mathbb F_q^n\), with \(n\ge 4\).  The organization of
this section follows the layer decomposition
\[
\Gr(2,V)=\bigsqcup_{k=2}^n \Gr_k(2,V).
\]
The high layers \(\Gr_k(2,V)\), \(k\ge 4\), have already been handled
uniformly for all prime powers \(q\).  The only part that requires separate
attention is the lower block
\[
\Gr_{\le 3}(2,V):=\Gr_2(2,V)\sqcup \Gr_3(2,V).
\]
For \(q>2\), this lower block admits a uniform frame-decomposition treatment.
For \(q=2\), the dimension-three local frame construction is too small, and
we instead use an explicit universal cycle in \(\mathbb F_2^3\), translated to
the cosets of \(V_3\) in \(V\).

In both cases, once the lower block is constructed, it can be glued to the
\(k=4\) layer through a common direction in \(V_3\subset V_4\).  The higher
layers are then glued successively through common directions.

\subsection{The High Layers}

For each \(k\ge 4\), Theorem~\ref{thm:index-k-frame-decomposition} gives a
frame partition \(\mathcal S_k\) of \(\Gr_k(2,V)\) such that the frame graph is
connected and
\[
\mathcal D(\mathcal S_k)=\{\text{all directions in }V_k\}.
\]
This applies for every prime power \(q\).

These high-layer decompositions are compatible with one another.  Indeed,
fix a direction, say \([e_1]\).  Since \([e_1]\subset V_k\) for every
\(k\ge 4\), each \(\mathcal S_k\) contains a frame whose direction set contains
\([e_1]\).  Therefore
\[
\mathcal S_{\ge 4}:=\bigcup_{k=4}^n \mathcal S_k
\]
has connected frame graph.  Moreover, since the top layer \(\mathcal S_n\)
has full direction set in \(V_n=V\), the union \(\mathcal S_{\ge 4}\) already
contains every direction in \(V\).

Thus, after the high layers have been assembled, the only remaining task is
to attach the lower block \(\Gr_{\le 3}(2,V)\) to \(\mathcal S_{\ge 4}\).
Because \(\mathcal S_4\) contains all directions in \(V_4\), any lower-layer
frame whose direction set contains a direction in \(V_3\) is adjacent to
\(\mathcal S_4\).

\subsection{The Lower Layers for \(q>2\)}

We first handle the lower block
\[
\Gr_{\le 3}(2,V)=\Gr_2(2,V)\sqcup \Gr_3(2,V)
\]
when \(q>2\).  Here \(\Gr_2(2,V)\) consists of the single plane
\(V_2=\langle e_1,e_2\rangle\), while \(\Gr_3(2,V)\) is organized by the
parallel classes corresponding to the \(1\)-subspaces of \(V_2\).

\begin{proposition}[Lower-layer frame partition for \(q>2\)]
\label{prop:lower_layers_q_gt_2}
Let \(V=\mathbb F_q^n\) with \(n\ge 4\) and \(q>2\).  Then
\(\Gr_{\le 3}(2,V)\) admits a partition \(\mathcal S_{\le 3}\) into local
frames.  Every frame in \(\mathcal S_{\le 3}\) has a direction lying in
\(V_3\).  Consequently, every frame in \(\mathcal S_{\le 3}\) is adjacent to
the high-layer decomposition \(\mathcal S_4\) of \(\Gr_4(2,V)\).
\end{proposition}

\begin{proof}
The parallel classes in \(\Gr_3(2,V)\) are indexed by the \(1\)-subspaces of
\(V_2=\langle e_1,e_2\rangle\).  Choose two of them,
\[
L_1=\langle e_1\rangle,\qquad L_2=\langle e_2\rangle,
\]
and reserve them to be treated together with the unique plane
\(V_2\in\Gr_2(2,V)\).

The remaining \(q-1\) directions in \(V_2\) are handled by the pair and triple
constructions from Section~4.  If \(q\) is odd, then \(q-1\) is even, so we
partition the remaining directions into pairs and apply the pair construction
to each pair.  If \(q\) is even and \(q>2\), then \(q-1\) is odd; we choose
one triple of remaining directions and apply the triple construction to it,
and then pair the remaining directions.  These constructions partition all
parallel classes in \(\Gr_3(2,V)\) except the two reserved classes
\(\mathcal A(L_1)\) and \(\mathcal A(L_2)\).

It remains to combine these two reserved classes with \(V_2\).  Set
\[
a_0=e_1+e_2,
\qquad W=\langle a_0\rangle.
\]
Then
\[
V_2=L_1\oplus W=L_2\oplus W,
\]
and the two reserved parallel classes are parametrized by
\[
\mathfrak X
=
\mathcal A(L_1)\sqcup\mathcal A(L_2)
=
\{\langle e_1,a+e_3\rangle,
  \langle e_2,a+e_3\rangle : a\in \langle a_0\rangle\}.
\]
We now partition \(\mathfrak X\cup\{V_2\}\) into frames.

\medskip
\noindent
\textbf{Case 1: \(q\) odd.}
In the pair construction, \(\mathfrak X\) is partitioned into
\[
\mathcal T_+,
\qquad
\mathcal T_-,
\qquad
\mathcal Q_a.
\]
We replace the two three-frames \(\mathcal T_+\) and \(\mathcal T_-\) by
\[
\mathcal T_+'
=
\{\langle e_1,e_3\rangle,
  \langle e_2,e_3\rangle,
  \langle e_1,e_2\rangle\},
\]
together with the four-frame
\[
\mathcal Q_{a_0}
=
\{\langle e_1,a_0+e_3\rangle,
  \langle e_1,-a_0+e_3\rangle,
  \langle e_2,a_0+e_3\rangle,
  \langle e_2,-a_0+e_3\rangle\}.
\]
Thus \(\mathcal T_+\cup\mathcal T_-\) is replaced by
\(\mathcal T_+'\cup\mathcal Q_{a_0}\), inserting
\(V_2=\langle e_1,e_2\rangle\) while preserving a partition of
\(\mathfrak X\cup\{V_2\}\).

\medskip
\noindent
\textbf{Case 2: \(q\) even and \(q>2\).}
In the pair construction, the relevant four-frames have the form
\[
\mathcal Q_a
=
\{\langle e_1,a+e_3\rangle,
  \langle e_1,a+a_0+e_3\rangle,
  \langle e_2,a+e_3\rangle,
  \langle e_2,a+a_0+e_3\rangle\},
\]
where \(a\in\langle a_0\rangle\) modulo the relation \(a\sim a+a_0\).
Choose \(\alpha\in\mathbb F_q\setminus\{0,1\}\).  We replace the two
four-frames \(\mathcal Q_0\) and \(\mathcal Q_{\alpha a_0}\), together with
\(V_2=\langle e_1,e_2\rangle\), by the following three three-frames:
\begin{align*}
\mathcal Q_0'&=
\{\langle e_1,e_3\rangle,
  \langle e_2,e_3\rangle,
  \langle e_1,e_2\rangle\},\\
\mathcal Q_1'&=
\{\langle e_1,a_0+e_3\rangle,
  \langle e_2,\alpha a_0+e_3\rangle,
  \langle e_2,(\alpha+1)a_0+e_3\rangle\},\\
\mathcal Q_2'&=
\{\langle e_2,a_0+e_3\rangle,
  \langle e_1,\alpha a_0+e_3\rangle,
  \langle e_1,(\alpha+1)a_0+e_3\rangle\}.
\end{align*}
These three-frames are disjoint and together cover exactly the planes of
\(\mathcal Q_0\cup\mathcal Q_{\alpha a_0}\cup\{V_2\}\).  Hence they give the
required replacement inside \(\mathfrak X\cup\{V_2\}\).

Combining the reserved-class construction with the pair and triple
constructions on the remaining parallel classes gives a frame partition
\(\mathcal S_{\le 3}\) of \(\Gr_{\le3}(2,V)\).  Every frame constructed above
lies in \(V_3\), so its direction set contains a direction in \(V_3\).  Since
\(\mathcal S_4\) contains all directions in \(V_4\), each such frame is
adjacent to some frame in \(\mathcal S_4\).
\end{proof}

The preceding proposition gives the lower-layer input for \(q>2\).  Combining
it with the high layers gives the following global frame decomposition.

\begin{theorem}[Global frame decomposition for \(q>2\)]
\label{thm:global_frame_decomposition}
Let \(V=\mathbb F_q^n\) with \(n\ge4\) and \(q>2\).  Then \(\Gr(2,V)\) admits
a frame partition \(\mathcal S\) such that
\begin{itemize}
\item the associated frame graph is connected, and
\item
\[
\mathcal D(\mathcal S)=\{\text{all directions in }V\}.
\]
\end{itemize}
\end{theorem}

\begin{proof}
Use Proposition~\ref{prop:lower_layers_q_gt_2} for the lower block
\(\Gr_{\le3}(2,V)\), and use Theorem~\ref{thm:index-k-frame-decomposition}
for each high layer \(\Gr_k(2,V)\), \(k\ge4\).  The high-layer union
\(\mathcal S_{\ge4}\) is connected by the discussion above.  Every lower-layer
frame is adjacent to \(\mathcal S_4\), hence to \(\mathcal S_{\ge4}\).  Thus
the full union
\[
\mathcal S=\mathcal S_{\le3}\cup\mathcal S_{\ge4}
\]
is connected.  Since \(\mathcal S_n\subset\mathcal S\) has full direction set
in \(V_n=V\), we have
\[
\mathcal D(\mathcal S)=\{\text{all directions in }V\}.
\]
\end{proof}

\subsection{The Lower Layers for \(q=2\)}

For \(q=2\), the uniform lower-layer frame construction above is not
available.  The obstruction is concentrated in dimension three, so we handle
that part by an explicit cycle.

Let \(V_3=\langle e_1,e_2,e_3\rangle\).  In \(V_3=\mathbb F_2^3\), consider
the cyclic sequence
\[
\begin{aligned}
C_0=\bigl(&[e_1],\,0,\,e_3,\,e_2,\,e_1,\,e_1+e_3,\,e_1+e_2,\,0,\,
[e_2],\,[e_1+e_3],\\
& e_2+e_3,\,e_1+e_2+e_3,\,e_3,\,[e_1+e_2+e_3]\bigr).
\end{aligned}
\]
This cycle encodes the \(14\) affine planes in \(V_3\): each cyclic
length-three window encodes a distinct affine plane, and all \(14\) affine
planes appear exactly once.  It contains the affine interface segment
\[
[e_1]\to 0
\]
and also the direction--direction segment
\[
[e_2]\to [e_1+e_3].
\]

For general \(n\ge4\), write
\[
V=V_3\oplus W.
\]
For each \(b\in W\), translate \(C_0\) by \(b\), translating affine-point
entries and leaving direction entries fixed.  The translated cycle encodes
exactly the affine planes contained in the affine \(3\)-space \(b+V_3\), with
linear part contained in \(V_3\).  As \(b\) ranges over \(W\), these cycles
cover all affine planes whose linear part lies in \(V_3\).

All translated copies contain the same direction--direction segment
\([e_2]\to [e_1+e_3]\), because direction entries are not translated.  Hence
the translated copies can be glued along this common pair.  We choose the
cuts away from the segment \([e_1]\to0\) in the copy \(C_0\), so the resulting
lower-layer cycle still contains the affine interface segment \([e_1]\to0\).

\subsection{Proof of Theorem~\ref{thm:affine_ucycle}}

We now prove the main affine-plane universal-cycle theorem.

Suppose first that \(q>2\).  By
Theorem~\ref{thm:global_frame_decomposition}, the whole Grassmannian
\(\Gr(2,V)\) admits a connected frame partition whose direction set contains
all directions in \(V\).  Applying the local universal-cycle constructions of
Section~3 to each frame, and then applying the connected gluing principle
(Theorem~\ref{thm:connected_gluing_principle}), gives a universal
triple-window cycle for all affine planes in \(V\).

Now suppose that \(q=2\).  The high layers \(\Gr_k(2,V)\), \(k\ge4\), are
handled by Theorem~\ref{thm:index-k-frame-decomposition}; as explained
above, these layers glue together through common directions.  The lower
block \(\Gr_{\le3}(2,V)\) is handled by the explicit translated cycle
constructed in the preceding subsection.  This lower cycle contains the
interface segment \([e_1]\to0\), while the high-layer cycle contains an
interface involving the direction \([e_1]\) by the local frame construction.
After reversing and translating the high-layer cycle if necessary, the two
interfaces can be matched.  Applying Lemma~\ref{lem:gluing_cycle}, or
equivalently the connected gluing principle, glues the lower cycle to the
high-layer cycle.

Thus in every case there exists a universal triple-window cycle encoding all
affine planes in \(\mathbb F_q^n\).  This proves
Theorem~\ref{thm:affine_ucycle}.

\subsection{Extension to $3$-Subspaces}

Finally, we extend the construction from affine planes to
\(3\)-dimensional subspaces of \(\mathbb F_q^n\), for all \(n\ge4\). The
argument separates the subspaces according to a fixed hyperplane and uses
the affine-plane construction on the complementary part.

\subsubsection{Base Case: $n=4$}

Let $N = \frac{q^4 - 1}{q - 1}$ and let $\alpha$ be a generator of
$\mathbb{F}_{q^4}^\times$. The cyclic sequence
\[
1 \to \alpha \to \alpha^2 \to \cdots \to \alpha^{N-1}
\]
runs through all $1$-subspaces of $\mathbb{F}_{q^4}\cong \mathbb{F}_q^4$.

Since $3$-subspaces of $\mathbb F_q^4$ are hyperplanes, i.e.\
codimension-$1$ subspaces, they are in bijection with $1$-subspaces via
duality. The multiplicative action of $\mathbb F_{q^4}^\times$ generated
by $\alpha$ is transitive on $1$-subspaces, and hence also on
$3$-subspaces. Therefore the sliding windows
\[
\langle \alpha^i,\alpha^{i+1},\alpha^{i+2}\rangle
=
\alpha^i \cdot \langle 1,\alpha,\alpha^2 \rangle
\]
run through all \(3\)-dimensional subspaces. Indeed, the stabilizer of
\(\langle 1,\alpha,\alpha^2\rangle\) under multiplication is
\(\mathbb F_q^\times\): if a larger element stabilized this
\(3\)-subspace, then the subspace would be a vector space over a proper
extension field of \(\mathbb F_q\), which is impossible in dimension \(3\)
inside \(\mathbb F_{q^4}\).
Thus this sequence forms a universal cycle whose direction set consists of
all \(1\)-subspaces. It also contains the consecutive independent pair
\([1]\to[\alpha]\), which will serve as the interface property in the
induction.

\subsubsection{Inductive Step}

Assume the statement holds for $n-1 \ge 4$, and write
\[
\mathbb{F}_q^n = V_{n-1} \oplus \langle e_n \rangle.
\]
Every $3$-subspace is either contained in $V_{n-1}$ or not.
We construct cycles on these two parts separately and then glue them along
a common pair of linear directions.

\medskip
\noindent
\textbf{(i) Subspaces contained in $V_{n-1}$.}
By the induction hypothesis, these subspaces have a universal cycle whose
direction set contains all \(1\)-subspaces of \(V_{n-1}\). We use the
induction in the slightly stronger form supplied by the base case: this
cycle also contains a consecutive pair of entries lying in \(V_{n-1}\).

\medskip
\noindent
\textbf{(ii) Subspaces not contained in $V_{n-1}$.}
Each such subspace determines a unique affine plane in \(V_{n-1}\). It may
be written in the form
\[
H = \langle a + e_n, u, v \rangle,
\qquad u,v \in V_{n-1},
\]
and the associated affine plane is
\[
a + \langle u,v \rangle \subset V_{n-1}.
\]
This gives a bijection between such \(3\)-subspaces and affine planes in
\(V_{n-1}\). By Theorem~\ref{thm:affine_ucycle}, they admit a universal
cycle whose direction set consists of all \(1\)-subspaces of \(V_{n-1}\).

\medskip
\noindent
\textbf{Gluing.}
Under this bijection, we identify the affine-plane cycle with a cycle on
$3$-subspaces by sending $[u]\mapsto [u]$ and $v\mapsto [v+e_n]$.

Thus all entries in both cycles are \(1\)-subspaces of
\(\mathbb F_q^n\). Those contained in \(V_{n-1}\) are the linear
directions, while those not contained in \(V_{n-1}\) represent affine
directions.

The outer cycle is obtained from the universal cycle on affine planes. By
the local construction on three-frames (Section~3), its building blocks
have the form
\[
[u_i]\to [u_{i+1}] \to w \to [u_{i+2}] \to [u_{i+3}],
\]
so this cycle contains consecutive pairs of linear directions
\[
[u]\to [v], \qquad [u],[v]\subset V_{n-1}.
\]
This property is preserved under the above identification, since linear
entries are unchanged.

Thus both cycles contain consecutive pairs of entries lying in
\(V_{n-1}\), i.e. consecutive pairs of linear directions. Since such a pair
occurs inside a window encoding a subspace, the two directions are linearly
independent. The group $\mathrm{GL}(V_{n-1})$ acts transitively on ordered
pairs of linearly independent $1$-subspaces, so we may apply a linear
automorphism of $V_{n-1}$ (extended to $V_n$) so that a common pair
$([u],[v])$ appears in both cycles.

Cutting both cycles at this segment and applying
Lemma~\ref{lem:gluing_cycle}, we obtain a universal cycle on all
$3$-subspaces of $\mathbb{F}_q^n$. This completes the induction and proves
Theorem~\ref{thm:3_subspace_ucycle}.

\bibliographystyle{amsalpha}  
\bibliography{reference}     

\end{document}